\documentclass[10pt]{amsart}

\usepackage{amssymb,latexsym,amsmath,amsfonts}
\usepackage[mathscr]{eucal}
\usepackage{mathrsfs}
\usepackage{pxfonts}
\usepackage{amscd}
\usepackage{amsthm}
\usepackage{amsxtra}
\usepackage[nointegrals]{wasysym}
\usepackage{bm}
\usepackage{calc}
\usepackage{float}
\usepackage[usenames,dvipsnames]{xcolor}
\usepackage{hyperref}

\hypersetup{colorlinks=true, citecolor=MidnightBlue, linkcolor=BrickRed}

\numberwithin{equation}{section}
\theoremstyle{definition}
\newtheorem{definition}{Definition}[section]
\theoremstyle{definition}

\newtheorem{remark}[definition]{Remark}

\theoremstyle{plain}
\newtheorem{theorem}[definition]{Theorem}
\newtheorem{lemma}[definition]{Lemma}
\newtheorem{cor}[definition]{Corollary}

\newcommand{\beas}{\begin{eqnarray*}}
\newcommand{\eeas}{\end{eqnarray*}}
\newcommand{\bes} {\begin{equation*}}
\newcommand{\ees} {\end{equation*}}
\newcommand{\be} {\begin{equation}}
\newcommand{\ee} {\end{equation}}
\newcommand{\bea} {\begin{eqnarray}}
\newcommand{\eea} {\end{eqnarray}}

\newcommand{\zt}{\zeta}

\newcommand{\om}{\omega}

\newcommand{\Om}{\Omega}
\newcommand{\D}{\mathbb{D}}

\newcommand{\zbar}{\overline z}
\newcommand{\wbar}{\overline w}

\newcommand{\wt}{\widetilde}
\newcommand{\cont}{\mathcal{C}}
\newcommand{\hol}{\mathcal{O}}

\newcommand{\sh}{\operatorname{sh}}

\newcommand{\Cn}{\mathbb{C}^n}

\newcommand{\C} {\mathbb{C}} 
\newcommand{\Z} {\mathbb{Z}}

\newcommand{\N} {\mathbb{N}}
\newcommand{\p}{\mathbb{P}^1}

\begin{document}

\title{On the dimension of Bergman spaces on $\p$}
\author{A.-K. Gallagher}
\address{Gallagher Tool \& Instrument LLC, Redmond, WA}
\email{anne.g@gallagherti.com}
\author{P. Gupta}
\address{Department of Mathematics, Indian Institute of Science, Bangalore}
\email{purvigupta@iisc.ac.in}
\author{L. Vivas}
\address{Department of Mathematics, The Ohio State University, Columbus, OH}
\email{vivas@math.osu.edu}
\thanks{P. Gupta and L. Vivas were partially supported by UGC CAS-II grant (GrantNo. F.510/25/CAS-II/2018(SAP-I)) and NSF grant 1800777, respectively.}

\begin{abstract}
Inspired by a result of Sz\H{o}ke, we give potential-theoretic characterizations of the dimension of the Bergman space of holomorphic sections of the restriction of a holomorphic line bundle on $\p$ to some open set $D\subset\p$.
\end{abstract}

\maketitle

\section{Introduction}\label{S:Intro}


The Bergman space, $A^2(\Omega)$, of an open set $\Omega\subset\mathbb{C}^n$ is the
vector space of $L^2(\Omega)$-integrable holomorphic functions on $\Omega$.
Through work of Carleson \cite{Ca67} and Wiegerinck \cite{Wi84} the dimension of the Bergman space of an open set in the complex plane is completely characterized by the polarity of the complement, see the equivalences $(a)$--$(c)$ in Theorem \ref{T:WC}. Wiegerinck \cite{Wi84} also constructs domains in  $\mathbb{C}^2$ whose Bergman spaces are nontrivial but finite dimensional; see \cite{PZ05,Dinew07,Jucha12,GaHaHe17,PZ17,Boud21} for further partial results on  the dimension of Bergman spaces of open sets in higher dimensions. 

The results on the dimension of the Bergman space of open sets in $\mathbb{C}$ may be consolidated as the following theorem.

\begin{theorem}[\cite{Ca67},\cite{Wi84},\cite{GaHaHe17},\cite{GaLeRa21}]\label{T:WC} Let $K\subsetneq\C$ be a closed subset. Then the following are equivalent.
\begin{itemize}
\item[$(a)$] $K$ is polar.
\item [$(b)$] $A^2(\C\setminus K)= \{0\}$.
\item [$(c)$] $\dim A^2(\C\setminus K)<\infty$.
\item [$(d)$] There exists no bounded $\psi\in\mathcal{C}^\infty(\mathbb{C}\setminus K)$ such that $\triangle\psi>0$ on $\mathbb{C}\setminus K$.
\end{itemize}
\end{theorem}
 That the Bergman space is either trivial or infinite dimensional, i.e., the equivalence of $(b)$ and $(c)$, is due to Wiegerinck, see \cite{Wi84}. Carleson shows  in \cite[Theorem 1, \S VI]{Ca67} that the Bergman space is nontrivial if and only if the logarithmic capacity of $K$ is positive, that is, the equivalence of $(a)$ and $(c)$. Property $(d)$ is an additional, potential-theoretic characterization of  open sets in $\mathbb{C}$ with finite dimensional Bergman spaces. A proof of implication $(c)\Rightarrow(d)$ may be found in \cite{GaHaHe17}, while implication $(d)\Rightarrow(c)$ is due to \cite[Prop. 5.1]{GaLeRa21}. 
 
 Recently, Sz\H{o}ke \cite{Sz20} considered the dimension problem
 in the setting of compact Riemann surfaces. He conjectures that, given a  holomorphic line bundle $L$ on a compact Riemann surface $M$,   the Bergman space, $A^2(D;L)$, of holomorphic sections of the restriction of  $L$  to an open set $D\subseteq M$ either coincides with the space of global holomorphic sections of $L$, or is infinite dimensional; see Subsection~\ref{SS:BSforSections} for the precise definition of $A^2(D;L)$.  Furthermore, he proves this conjecture in the case when $M$ is $\p$,  the Riemann sphere.
 
 \begin{theorem}\cite{Sz20}\label{T:Sz}
   Let $L$ be a holomorphic line bundle  on $\p$, and $D\subset\p$ be an open set. Then $A^2(D;L)$ is either equal to $\Gamma(\p;L)$, i.e., the space of global holomorphic sections on $\p$, or infinite dimensional.
 \end{theorem}

Sz\H{o}ke's proof is a modification of Wiegerinck's proof in \cite{Wi84}, and relies on the relatively simple algebraic structure of holomorphic line bundles on $\p$. While this proof is hard to generalize to other compact Riemann surfaces, Sz\H{o}ke gives a partial result for general compact Riemann surfaces. However, he does not provide any potential-theoretic characterizations, akin to Theorem~\ref{T:WC}, for the general case; see the remarks above Proposition 1.3 in \cite{Sz20}.  An objective of this paper is to fill this gap for $\p$. In fact, our main result is a complete analogue of Theorem~\ref{T:WC} for Bergman spaces of holomorphic sections on open subsets of $\p$. 
 
\begin{theorem}\label{T:main} Let $\Pi:L\rightarrow \p$ be a holomorphic line bundle. 
Suppose that $K\subsetneq\p$ is a compact subset. Then the following are equivalent.
	\begin{itemize}
\item[$(a)$] $K$ is polar.
\item[$(b)$] $A^2(\p\setminus K;L)=\Gamma(\p;L)$.
\item[$(c)$] $\dim A^2(\p\setminus K;L)<\infty$. 
\item[$(d)$] There exists no bounded function $\psi\in\mathcal{C}^\infty(\p\setminus K)$ such that $i\partial\bar{\partial}\psi\geq \omega$ on $\p\setminus K$ for some volume form $\omega$ on $\p$.
\end{itemize}  
\end{theorem}  

The implications $(a)\Rightarrow(b)\Rightarrow(c)\Rightarrow(a)$ are proved using some basic results in potential theory. 
In particular, our proof of $(c)\Rightarrow(a)$ in Subsection~\ref{SS:CT} is inspired by Wiegerinck's proof in \cite{Wi84}, but uses Carleson's construction of a nontrivial $L^2$-integrable holomorphic function as seen in \cite[Theorem 9.5, Ch. 21]{Ca67}. We include this proof of the implication $(c)\Rightarrow(a)$, because it does not rely on partial differential equations techniques, unlike our proof of $(c)\Rightarrow(d)\Rightarrow(a)$.

As in the proof of Theorem~\ref{T:WC}, the proof of the implication  $(c)\Rightarrow(d)$ in
 Theorem~\ref{T:main} is done by using H\"ormander's weighted $L^2$-method for solving the Cauchy--Riemann equations. 
 We note that the equivalence of $(a)$ and $(d)$ can be paraphrased as: $\p\setminus K$ is nonpolar if and only if, for every Hermitian metric $h$ on $L$ and  volume form $\omega$ on $\p$, there is a bounded function  $\psi\in\mathcal{C}^\infty(\p\setminus K)$ satisfying
  \begin{align*}
     i\Theta_{e^{-\psi}h}\geq \omega\;\;\text{ on }\p\setminus K.
  \end{align*}
That is, whenever $K$ is nonpolar, the given Hermitian metric $h$ on $L$ may be twisted such that the newly obtained Hermitian metric on $L|_{D}$ has positive curvature, see \cite[Theorem 3.11.2]{NaRa11} for a related result on general open Riemann surfaces.  
This positivity lets one use 
H\"ormander's method successfully; the boundedness of $\psi$ ensures that the resultant estimates are for the Bergman spaces associated to the given Hermitian metric $h$ on $L$.

The proof of $(d)\Rightarrow(a)$ in Theorem~\ref{T:main} is also motivated by the proof of the same implication in Theorem~\ref{T:main}. In the latter, one uses the potential function associated to the equilibrium measure of a compact, nonpolar set to construct the strictly subharmonic, bounded weight function $\psi$. This construction may be done in a local chart to prove the implication
 $(d)\Rightarrow(a)$ in Theorem~\ref{T:main}. 
In fact, the proofs of the equivalences in Theorem~\ref{T:main} are all done in a local chart. However, these proofs, except for the direct proof of $(c)\Rightarrow(a)$  in Subsection~\ref{SS:CT},  do not depend on the particular structure of $L$ and $h$ so that an extension of Theorem~\ref{T:main} to compact Riemann surfaces seems feasible. In fact, we will consider this question for compact Riemann surfaces in a forthcoming paper. 

As a byproduct of the proof of the implication $(c)\Rightarrow(d)$ of Theorem~\ref{T:main}, one obtains a complete description of the dimension of the weighted Bergman space $A^2_{e^{-\psi}}(\Omega)$, where $\Omega\subset\mathbb{C}$ is an open set, and $\psi$ is a subharmonic function on $\Omega$. This is a combination of Corollary~\ref{C:LeThesisminusK}, which is a
 version of Theorem~\ref{T:WC} for $A^2_{e^{-\psi}}(\Omega)$, and the complete characterization of the dimension of $A_{e^{-\psi}}(\mathbb{C})$ due to Borichev, Le, and Youssfi; see \cite[Theorem 2.6]{Le21}.
 
In \cite{Wi84}, Wiegerinck gives examples of domains in $\C^2$ that have nontrivial, but finite dimensional Bergman spaces. These examples immediately show that the dichotomy conjectured for compact Riemann surfaces does not hold for the Bergman spaces of holomorphic sections of a specific holomorphic line bundle, $L=\hol(-3)$, on the complex projective space $\mathbb{P}^2$. Motivated by Wiegerinck's construction, we produce examples in $\mathbb{P}^2$ to show that this dichotomy is absent irrespective of the choice of holomorphic line bundle on $\mathbb{P}^2$. 

The paper is structured as follows. In Section~\ref{S:prelim}, we detail the required background material and notation on Bergman spaces, potential theory, and holomorphic line bundles on $\p$. The first four subsections of Section~\ref{S:proofs} contain the proofs of the equivalences of Theorem~\ref{T:main}. This is followed by Subsection~\ref{SS:Le}, where the results on the dimension of weighted Bergman spaces for open sets in $\mathbb{C}$ are given.  In Section \ref{S:counter}, we show by example that for each holomorphic line bundle on $\mathbb{P}^2$, there is a domain in $\mathbb{P}^2$ such that the Bergman space of the corresponding holomorphic sections does not equal the space of global holomorphic sections but is finite dimensional.

\noindent {\bf Acknowledgements.}  The authors are grateful to Peter Pflug for his helpful comments on Lemma~\ref{L:removable}.

\section{Background and preliminaries}\label{S:prelim}


\subsection{Bergman spaces in $\Cn$ and potential theory in $\mathbb{C}$} Let $\Om\subseteq \Cn$ be an open subset, and $\lambda$ be the standard volume form on $\C$. For a positive function $\phi$ on $\Om$, the {\em weighted Bergman space} of $\Om$ with weight $\phi$ is the Banach space
	\bes
		A^2_\phi(\Om):=\left\{f\in\hol(\Om):||f||:=\left(\int_D|f(z)|^2\phi(z)\lambda(z)
	\right)^{1/2}<\infty\right\}.
	\ees
For $\phi\equiv 1$, the space $A^2_\phi(\Om)$ is denoted by $A^2(\Om)$, and referred to as the Bergman space of $\Om$. 

A set $K\subset\C$ is said to be {\em polar} if there is a nonconstant subharmonic function $s$ on $\C$ such that $K\subseteq\{z\in\C:s(z)=-\infty\}$. Polar sets admit an alternate characterization via logarithmic potential theory. For a finite Borel measure $\nu$ with compact support in $\C$, its {\em potential} is the function $p_\nu:\C\rightarrow [-\infty,\infty)$ given by 
	\bes
		p_\nu(z)=\int_\C\ln|z-w|\,d\nu(w),\qquad z\in\C.
	\ees
The {\em energy} of $\nu$ is the quantity
	\bes
		I(\nu)=\int_\C p_{\nu}(z)\ d\nu(z).	
	\ees
The {\em (logarithmic) capacity} of $K\subset \C$ is defined as 
	\bes
		\operatorname{cap}(K):=\sup\{e^{I(\nu)}:\nu\ \text{is a Borel probability measure with compact support in}\ K\}.
	\ees
For a compact nonpolar set $K\subset\C$, it is known that there is a unique Borel probability measure $\nu_K$ supported on $K$ such that 
$$I(\nu_K)=\sup\{\nu:\nu\ \text{is a Borel probability measure on}\ K\}.$$
This measure is called the {\em equilibrium measure} of $K$.  
The following result relates the notion of polarity, local polarity, and capacity. 

\begin{theorem}\label{T:polar} Let $K\subset\C$ be a Borel subset. Then the following are equivalent. 
\begin{itemize}
\item [$(i)$] $K$ is polar. 
\item [$(ii)$] There is an open set $G$ containing $K$, and an $s\in\sh(G)$, nonconstant on any component of $G$, such that $K\subset \{x\in G:s(x)=-\infty\}$.
\item [$(iii)$] For every $\zt\in \C$, there is a connected open neighborhood $V_\zt\subseteq\C$ of $\zt$ and a nonconstant function $s_\zt\in\sh(V_\zt)$ such that $K\cap V_\zt\subseteq\{x\in V_\zt:s_\zt(x)=-\infty\}$.
\item [$(iv)$] $\operatorname{cap}(K)=0$.
\end{itemize}
\end{theorem}
For the equivalences of $(i)$ and $(ii)$, $(i)$ and $(iii)$, and $(i)$ and $(iv)$, see for instance Proposition~$5.5$, Lemma~$5.6$ and Theorem~$7.5$, respectively, in \cite[Ch.~21]{Co95}.


\subsection{Bergman spaces of holomorphic sections}\label{SS:BSforSections}
 Let $M$ be a complex manifold. Given a holomorphic line bundle $\Pi:L\rightarrow M$ and an open set $D\subset M$,
the space of holomorphic sections of $L|_D$ is denoted by $\Gamma(D;L)$. 

Let $\om$ be a volume form on $M$. Given a holomorphic line bundle $\Pi:L\rightarrow M$, and a smooth Hermitian metric $h$ on $L$, 
we define the Bergman space of sections of $(L|_D,h)$ as
	\bes
		A^2_{h,\om}(D;L)=\left\{s\in\Gamma(D;L):||s||:=\left(\int_D h(s,s)\,\omega\right)^{1/2}<\infty\right\}.
	\ees
$A^2_{h,\om}(D;L)$ is a reproducing kernel Hilbert space. 
If $M=\Cn$, then $\Gamma(D;L)=\hol(D)$ and $h(s,s)\omega=|s|^2\phi\lambda$, for some positive function $\phi$ on $\Cn$. In this case, $A^2_{h,\om}(D;L)$ is the weighted Bergman space $A^2_\phi(D)$. 

\begin{lemma}\label{L:choice} Let $M, D, L, h,\om$ be as above. Suppose that $M$ is compact. Then, as a vector space, $A^2_{h,\om}(D;L)$ is independent of the choices of $h$ and $\omega$. In particular, $A^2_{h,\omega}(M;L)=\Gamma(M;L)$.
\end{lemma}
\begin{proof} Given any two smooth volume forms, $\omega$ and $\omega'$, on $M$, and any two smooth Hermitian metrics, $h$ and $h'$, on $L$, there exist smooth positive functions, $f$ and $g$, on $M$, such that $\om=f\om'$ and $h=gh'$. Thus, the Bergman spaces of the sections of $L|_D$ are isomorphic for different choices of the volume form on $M$ and Hermitian metric on~$L$.

The final claim holds since $h(s,s)$ is a continuous positive function on $M$ for any choice of $h$ and $s$. 
\end{proof}

\begin{remark} In this paper, $M$ is either $\mathbb P^1$ or $\mathbb P^2$. Thus, in view of Lemma~\ref{L:choice}, $A^2_{h,\om}(D;L)$ is simply denoted by $A^2(D;L)$. 
\end{remark}

\subsection{Line bundles on $\p$} We recall certain standard facts about $\p$. Given any $q\in\p$, one may choose coordinates $\zt=[\zt_0:\zt_1]$ on $\p$ so that $q=[0:1]$. In this case, we often identify $U^z=\p\setminus\{q\}$ with the complex plane $\C$ via
	\bes
		\varphi_z:\zt=[\zt_0:\zt_1]\mapsto z=
							\frac{\zt_0}{\zt_1},
	\ees
and refer to $z$ as the local coordinate on the affine chart $\p\setminus \{q\}$. On occasion, the chart $(U^{1/z},\varphi_{1/z})$ given by $\varphi_{1/z}:\zt\mapsto \zt_1/\zt_0$ on $U^{1/z}=\p\setminus\{[1:0]\}$ will also be used. In this case, we will use the coordinate $w$ on $\varphi_{1/z}(U^{1/z})=\C$.

A smooth volume form $\om$ on $\p$ is of the form
		$f\,\om_{\text{FS}}$, 
where $\om_{\text{FS}}$ is the Fubini--Study volume form, and $f$ is a smooth positive function on $\p$. Thus, in the local coordinate $z$, 
	\bes
		\om(z)=f(z)\,\dfrac{idz\wedge d\zbar}{2(1+|z|^2)^2},\quad z\in\C,
	\ees
where $f:\C\rightarrow(0,\infty)$ is smooth, and $f\circ\varphi_z$ admits a smooth positive extension to $\p$. Given a smooth function $\psi$ on an open set in $\mathbb P^1$, $\partial\overline\partial \psi$ is the unique smooth $(1,1)$-form on $\p$ satisfying $(\varphi_z)_*\left(\partial\overline\partial \psi\right)(z)=\frac{\partial^2(\psi\circ\varphi_z)}{\partial z\partial\zbar}(z)\, d\zbar\wedge dz$ for all $z\in U^z$. 	 

Next, any holomorphic line bundle $\Pi:L\rightarrow\p$ is of the form $\hol(k)$, for some $k\in\Z$, where $\hol(k)$ is the line bundle associated to the divisor $k\{p\}$ for any fixed $p\in\p$. 
In the local coordinate $z$, a global section of $\hol(k)$ is given by a pair, $s=(s_1,s_2)$, of holomorphic functions on $\C$ such that $s_1(z)=z^{k}s_2(1/z)$ if $z\in\C^*$. Similarly, a Hermitian metric on $L=\hol(k)$ is given by a pair, $h=(h_1,h_2)$, of smooth positive functions on $\C$, such that $h_1(z)=|z|^{-2k}h_2(1/z)$ if $z\in\C^*$. Owing to the relationship between $h_1$ and $h_2$, $\partial\overline{\partial} h_1(z)=\partial\overline{\partial} h_2(1/z)$, $z\in\C^*$, where $\partial\overline\partial f(z)=\frac{\partial^2f}{\partial z\partial\zbar}(z)\, d\zbar\wedge dz$ for any $\cont^2$-smooth function $f$ on $\C$. Thus, there is a smooth $(1,1)$-form $\Theta_h$ on $\p$ such that $(\varphi_z)_*\Theta_h=\partial\overline{\partial} h_1$ and $\left(\varphi_{1/z}\right)_*\Theta_h=\partial\overline{\partial} h_2$. The form $\Theta_h$ is referred to as the curvature of the Hermitian metric $h$. Note that it suffices to specify $s_1$, $h_1$ and $\partial\overline\partial h_1$, which is the convention we will employ. Lastly, recall the following description of the space of global sections of $\hol(k)$, $k\in\Z$.  

\begin{lemma}\label{L:Bergman} Let $k\in\Z$. In the local coordinate $z$, 
	\bes
		A^2(\p;\hol(k))=\Gamma(\p;\hol(k))=\begin{cases}
				\{\text{polynomials of degree at most}\ k\},& \text{if}\ k\geq 0,\\
				\{f\equiv 0\},& \text{if}\ k< 0.	
				\end{cases}
	\ees
In particular, $\dim A^2(\p;\hol(k))=\dim\Gamma(\p;\hol(k))=\max\{0,k+1\}$. 
\end{lemma}

While Lemma~\ref{L:Bergman} is a standard result, see \cite[p. 81]{Varolin11}, we direct the reader to the remarks before Corollary~\ref{C:LeThesisminusK} for a proof in the case $k\geq -2$. 

\subsection{Potential theory on $\p$} Given an open set $U\subseteq\p$, and an upper semicontinuous function $s:U\rightarrow [-\infty, \infty)$ with $s\nequiv -\infty$, $s$ is said to be {\em subharmonic on $U$} if, for every coordinate chart $(V,\varphi)$ that intersects $U$, $s\circ\varphi^{-1}$ is subharmonic on $\varphi(U\cap V)\subset \C$. The set of all subharmonic functions on $U$ is denoted by $\sh(U)$.

\begin{definition}\label{D:polarP} Given a set $K\subset\p$, $K$ is said to be {\em polar} if for each $\zt\in\p$, there is an open neighborhood $V_\zt\subset\p$ of $\zt$ and a function $s_\zt\in\sh(V_\zt)$ such that $K\cap V_\zt\subset\{x\in V_\zt:s_\zt(x)=-\infty\}$.
\end{definition}

\begin{lemma}\label{L:polar1} Let $K\subset\p$ be compact. For any fixed $q=[0:1]\in\p$, $K$ is polar in $\p$ if and only if $K_*:=\varphi_z(K)$ is polar in $\C$. 
\end{lemma}
\begin{proof} First, suppose that $K$ is polar. Let $\zt_*\in\C$ and $\zt=\varphi_z^{-1}(\zt_*)$. Then it is clear that
	\bes
		K_*\cap U_{\zt_*}\subset\{x\in U_{\zt_*}:s_{\zt_*\,}(x)=-\infty\},
	\ees
where $U_{\zt_*}=\varphi_z(V_\zt\setminus\{q\})$, and $s_{\zt_*}=s_\zt\circ\varphi_z^{-1}$. Thus, $K_*$ is polar in $\C$. 

Next, assume that $K_*$ is polar. By a similar reasoning as above, we obtain that for every $\zt\in\p\setminus\{q\}$, there exist $V_\zt\subset\p$ and $s_\zt\in\sh(V_\zt)$ that satisfy the condition in Definition~\ref{D:polarP}. It remains to produce such a pair $(V_\zt,s_\zt)$ for $\zt=q$. Since subharmonicity is preserved by biholomorphic maps, it follows that the set $K_{**}:=\varphi_{1/z}\circ \varphi_z^{-1}(K_*)$ is polar in $\C^*$. In fact, by Theorem~\ref{T:polar}, $K_{**}$ is polar in $\C$. On the other hand, $\{0\}=\varphi_{1/z}(\{q\})$ is also polar in $\C$. Now, using the fact that the union of two polar sets in $\C$ is polar in $\C$, see \cite[Lemma~5.6, Ch.~21]{Co95}, $\varphi_{1/z}(K)=K_{**}\cup\{0\}$ is polar in $\C$, i.e., there is an $s\in\sh(\C)$ such that $\varphi_{1/z}(K)\subset\{x\in\C:s(x)=-\infty\}$. To complete the proof, we simply observe that $V_q=\varphi_{1/z}^{-1}(\C)$ and $s_q=s\circ\varphi_{1/z}$ satisfy the condition required at $q\in K$ in Definition~\ref{D:polarP}.  
\end{proof}


\section{Proof(s) of the main result}\label{S:proofs}
 By Lemma~\ref{L:Bergman}, $(b)\Rightarrow(c)$. In Subsections~\ref{SS:a=>b}--\ref{SS:(d)=>(a)}, we prove the implications $(a)\Rightarrow(b)$, $(c)\Rightarrow(a)$, $(c)\Rightarrow(d)$, and $(d)\Rightarrow(a)$, respectively. This section concludes with the characterization of weighted Bergman spaces, $A^2_{e^{-\psi}}(\Omega)$, for subharmonic $\psi$ and $\Omega\subset\mathbb{C}$ open.

\subsection{Proof of $\mathbf{(a)}$ implies $\mathbf{(b)}$}\label{SS:a=>b} 
In order to prove that $A^2(\p\setminus K;L)=\Gamma(\p;L)$ whenever $K\subset\p$ is a compact polar set, we first prove a version of this result for sets in $\mathbb{C}$ in Lemma~\ref{L:removable} below. 
A readily available result in the literature is that  $A^{2}(\Om)=A^2(\Om\setminus X)$ for any open set $\Om\subset\C$ and polar, compact set $X\subset \Om$, see for instance Theorem~9.5 in \cite[Ch.~21]{Co95}. 
We present a generalization of this result to the case when $X$ is a polar and relatively closed subset of $\Om$.\footnote{After the first draft of this manuscript appeared on the arXiv, Pflug brought to our attention that Lemma~\ref{L:removable} has appeared in \cite{Si82}, where Siciak attributes it to Sakai and Skwarczy{\' n}ski. Since neither \cite{Si82} nor the source cited therein are easily accessible, we give a proof of Lemma~\ref{L:removable} in this work. This proof is simpler than our original one, and follows an argument suggested by Pflug. This argument resembles that of Siciak's.}

\begin{lemma}[\cite{Si82}]\label{L:removable} 
Let $\Om\subseteq\C$ be an open subset, and $X\subset \C$ be a closed polar subset. Then, $A^2(\Om\setminus X)=A^2(\Om)$, i.e., for any $f\in A^2(\Om\setminus X)$, there is an $F\in A^2(\Om)$ such that $F|_{\Om\setminus X}=f$. 
\end{lemma}
\begin{proof} The proof uses the fact that for any $a\in \Om\cap X$, there exists an $r_a>0$ such that $\overline{\D(a,r_a)}\subset\Om$, and $b\D(a,r_a)\cap X=\emptyset$, see Theorem 7.3.9 in \cite{ArGa12}.

Now, let $D_a:=\D(a,r_a)$ and $X_a:=\D(a,r_a)\cap X$, where $r_a$ is as above. Then, $X_a$ is a compact subset of $D_a$. Using the result for compact, polar sets mentioned before the statement of this lemma, for any $f\in A^2(\Om)$, there is an $F_a\in A^2(D_a)$ such that $F_a|_{D_a\setminus X_a}=f|_{D_a\setminus X_a}$. For any $a,b\in X$ such that $D_a\cap D_b\neq \emptyset$, we have that $(D_a\cap D_b)\setminus X$ is a nonempty open set because $X$ is a closed and polar set. Thus, it follows from the identity theorem that $F_a=F_b$ on $D_a\cap D_b$. Hence, the following function is well-defined and holomorphic on $\Om$:
	\bes
		F(z)=\begin{cases}
				F_a(z),&\ z\in D_a,\\
					f(z),&\ z\in\Om\setminus X.
			\end{cases}
	\ees
 Moreover, since $f$ and $F$ differ only on the zero measure set $X$, it follows that $||F||_{A^2(\Om)}=||f||_{A^2(\Om\setminus X)}$.
\end{proof}

We are now set to prove that $(a)$ implies $(b)$ in Theorem~\ref{T:main}. 

\begin{proof}[Proof of $(a)\Rightarrow(b)$ in Theorem~\ref{T:main}]
Suppose that $K\subsetneq \p$ is polar. If $K$ is empty, then $(b)$ follows from $(a)$, see Lemma~\ref{L:Bergman}. Thus, we assume that $K$ is a nonempty polar set, and fix a $q\in K$. With $q=[0:1]$, we fix the Hermitian metric on $L=\hol(k)$ as $h_1(z)={1}/(1+|z|^2)^k$ in the local coordinate $z$. To account for the volume form in the local coordinate $z$, we set
\be\label{E:metric}
		\phi_k(z):=\frac{1}{(1+|z|^2)^{k+2}},\qquad z\in\C. 
\ee
Setting $D_*:=\C\setminus K_*$, it then follows that  $A^2(D;\hol(k))$ is isomorphic to $A^2_{\phi_k}(D_*)$.

 Next, we show that $A^2_{\phi_k}(D_*)=A^2_{\phi_k}(\C)$. For this, let $f\in A^2_{\phi_k}(D_*)$. For any $R>0$, set $D_*(R)=D_*\cap \D(0;R)$ and $f_R=f\cdot\chi_{_{D_*(R)}}$. Then, since 
	\bes
		\phi_k(z)\geq \begin{cases}	
			 1,& \text{if}\ k<-2,\\
			(1+R^2)^{-k-2},& \text{if}\ k\geq -2,
		\end{cases}\qquad z\in \D(0;R),
	\ees
it follows that $f_R\in A^2(D_*(R))$. However, $K_*$ is a closed polar subset of $\C$. 
Thus, by Lemma~\ref{L:removable}, there exists an $F_R\in A^2(\D(0;R))$ such that $F_R|_{D_*(R)}=f_R$. We abuse notation, and assume that $F_R$ is defined on $\C$ by setting it to be $0$ on 
$\C\setminus \D(0;R)$.  Since $$F_S|_{D_*(R)}=F_{R}|_{D_*(R)}=f|_{D_*(R)}\;\;\text{ whenever }\;\;S\geq R>0,$$
it follows from the identity theorem that $F_S|_{\D(0;R)}=F_R|_{\D(0;R)}$ for all $S\geq R>0$. 
Thus, the sequence $\{F_N\}_{n\in\N}$ admits a pointwise limit, say $F$, on $\C$. Moreover, since $F|_{\D(0;N)}=F_N$ for all $N\in\N$, $F\in\hol(\C)$ and $F|_{D_*}=f$. In particular, $F$ is measurable so that
\begin{align*}
  \|F\|_{L^2_{\phi_k}(\mathbb{C})}=\|F\|_{L^2_{\phi_k}(D_{*})}=\|f\|_{L^2_{\phi_k}(D_{*})}<\infty
\end{align*}
follows. Therefore, $A^2_{\phi_k}(D_*)$ is isomorphic to $A^2_{\phi_k}(\mathbb{C})$.  Finally, note that
$A^2_{\phi_k}(\C)$ is isomorphic to $A^2(\p;\hol(k))$, which is $\Gamma(\p;\hol(k))$, by Lemma~\ref{L:Bergman}. Thus, $(a)$ implies $(b)$ in Theorem~\ref{T:main}.
\end{proof}


\subsection{Proof of $\mathbf{(c)}$ implies $\mathbf{(a)}$ via the Cauchy transform approach}\label{SS:CT}

\begin{proof}
We prove the implication by contraposition, i.e., we prove that if $K\subsetneq\p$ is nonpolar, then $A^2(\p\setminus K;\hol(k))$ is infinite dimensional. As in the proof of $(a)\Rightarrow(b)$ in Subsection~\ref{SS:a=>b}, we assume that $q=[0:1]\in K$ and fix the Hermitian metric $h_1(z)=1/(1+|z|^2)^k$ on $\hol(k)$. Then, it suffices to show that $A^2_{\phi_k}(D_*)$ is infinite dimensional, where $\phi_k$ is as in \eqref{E:metric}, and $D_*=\C\setminus K_*$.

\noindent {\bf Case 1.} Suppose that $k\geq -2$. Then, since $\phi_k(z)\leq 1$ for all $z\in \C$, $A^2(D_*)\subseteq A^2_{\phi_k}(D_*)$. Since, by Lemma~\ref{L:polar1}, $\C\setminus D_*=K_*$ is nonpolar, Theorem~\ref{T:WC} yields that $A^2(D_*)$, and therefore, $A^2_{\phi_k}(D_*)$ is infinite dimensional. 

\noindent {\bf Case 2.} Suppose that $k<-2$. In this case,  $A^2_{\phi_k}(D_*)\subseteq A^2(D_*)$, so Theorem~\ref{T:WC},  do not directly yield our claim. However, we use the techniques of Carleson and Wiegerinck to produce infinitely many independent functions in $A^2_{\phi_k}(D_*)$. 

First, we recall Carleson's construction of a nontrivial function in $A^2(D_*)$; see \cite[Theorem~9.5, Ch.~21]{Co95} for a detailed exposition. Set $K_*:=\varphi_z(K)=\C\setminus D_*$. Then, $K_*$ is a Borel nonpolar set in $\C$. Thus, it contains a compact set, say $E$, with positive logarithmic capacity, see for instance \cite[Theorem 7.5, Ch.~21]{Co95}. Let $E_1,E_2$ be disjoint compact subsets of $E$, each of which has positive logarithmic capacity. For $j\in\{1,2\}$, let $\mu_j$ be the equilibrium measure of $E_j$, and set $\mu=\mu_1-\mu_2$. Set $f$ to be the Cauchy transform of $\mu$, i.e., 
	\begin{align}\label{E:Cauchytransform}
		f(z)=\int_E\frac{d\mu(\xi)}{\xi-z}.	
	\end{align}
Then, $f$ is analytic on $\C_\infty\setminus E$, with $f(\infty)=f'(\infty)=0$. Carleson further shows that  $f\in A^2(\C\setminus E)$ which is contained in $A^2(D_*)$. 

We next use Wiegerinck's technique to produce a sequence of linearly independent functions $\{g_j\}_{j\in\mathbb{N}}\subset A^2(\C\setminus E)$ such that $g_j$ vanishes at $\infty$, and its order of vanishing at $\infty$ is at least $j$. Assuming the existence of this sequence for the moment, we claim that $g_j\in A^2_{\phi_k	}(D_*)$ for all $j\geq -k$. To see this, fix a $j\geq -k$, and note that there exist $\{c_{j,\ell}\}_{\ell\geq j}\subset\C$ and $R>0$ so that
	\bes
		g^j(z)=\sum_{\ell=j}^\infty c_{\ell,j}\, z^{-\ell}
			\qquad \text{for}\ |z|>R.
	\ees
 Thus, since $j\geq -k$ and $k<-2$, it follows that   
	\begin{align*}
		\int_{D_*}|g_j(z)|^2&\frac{idz\wedge d\zbar}{2(1+|z|^2)^{k+2}}\\
		&\leq 
		(1+R^2)^{-2-k}\int_{D(0;R)\setminus E}|g_j(z)|^2\,\lambda(z)
			+\int_{|z|>R}|g_j(z)|^2\frac{idz\wedge d\zbar}{2(1+|z|^2)^{k+2}}\\
		&\leq	||g_j||^2_{A^2(\C\setminus E)}+2\pi \int_{r>R}\sum_{\ell=j}|c_{\ell,j}|^2r^{-2\ell} r(1+r^2)^{-2-k} dr<\infty. 
	\end{align*}
This gives the infinite dimensionality of $A^2_{\phi_k}(D_*)$. It remains to construct the sequence $\{g_j\}_{j\in\N}$.

\noindent{\em Subcase 1. Suppose that $f$ in \eqref{E:Cauchytransform} is rational.} Since $f$ is $L^2$-integrable on $\C\setminus E$, but not on $\C$, $E$ must have positive Lebesgue measure. In this case, the function 
	\bes	
		g(z)=\int_E\frac{\lambda(\xi)}{\xi-z}
	\ees
is bounded and analytic on $\C\setminus E$, with $g(\infty)=0$ and $g'(\infty)=-\lambda(E)$, see \cite[Page 2]{Ga72}. Thus, $g^j\in A^2(\C\setminus E)$ for all $j\in\N$, and each $g^j$ has order of vanishing $j$ at $\infty$. In this case, we set $g_j=g^j$.   

\noindent{\em Subcase 2. Suppose that $f$ in \eqref{E:Cauchytransform} is not rational.} Expanding $f$ in a Laurent series around $\infty$, one gets
	\bes
		f(z)=\sum_{\ell=p}^\infty c_\ell\,z^{-\ell},\qquad \text{ for some }p\geq 2, c_p\neq 0. 
	\ees
Now, we produce a nontrivial function $g\in A^2(\C\setminus E)$ whose Laurent expansion at $\infty$ does not contain any terms in $z^{-1}$,\ldots,$z^{-p}$. Let $z_1,\ldots,z_{p+1}$ be distinct points in $\C\setminus E$, and
	\bes
		g(z)=\sum_{\ell=1}^{p+1}b_\ell\:\frac{(f(z)-f(z_\ell))}{z-z_\ell}.
	\ees
Then, expanding $g$ as a Laurent series around $\infty$, one obtains $g(z)=\sum_{\ell=1}^\infty a_\ell\,z^\ell$, where 
	\be\label{E:coeff}
		a_m=\sum_{\ell=1}^{p+1}-b_\ell f(z_\ell)\,z_\ell^{m-1},\qquad		
			m\in\{1,\ldots,p\}.
	\ee
Here, the constants $b_1,\ldots,b_{p+1}$ are chosen to solve the homogeneous system of $p$ linear equations obtained by setting $a_m=0$, $m\in\{1,\ldots,p\}$. Moreover, $g$ cannot be trivial, else $f$ will be rational. Thus, $g$ has the desired properties. 

Setting $g_1=\cdots =g_p=f$ and $g_{p+1}=g$, we construct $g_j$ inductively for $j\geq p+1$ by repeating the above procedure for $g_{j-1}$ in place of $f$. This completes the construction of the sequence in all cases, and hence, the proof of Theorem~\ref{T:main}. 
\end{proof}


\subsection{Proof of $\mathbf{(c)}$ implies $\mathbf{(d)}$}\label{SS:dbar}
The proof of this implication follows from a similar result for sets in $\mathbb{C}$.

\begin{lemma}\label{L:GaHaHemodification}
  Let $\Omega\subset\mathbb{C}$ be an open set and $A_{e^{-\psi_1}}^2(\Omega)$ be the weighted Bergman space for some $\psi_1\in\mathcal{C}^\infty(\Omega,\mathbb{R})$. Suppose that
  there exists a subharmonic function $\psi_2\in\mathcal{C}^\infty(\Omega,\mathbb{R})$  such that it is bounded above, $\psi_1+\psi_2$ is subharmonic on $\Omega$, and
  $$\triangle\psi_2 >0\;\;\text{on }\overline{U}$$
for an open set $U\Subset\Omega$.  Then $A_{e^{-\psi_1}}^{2}(\Omega)$ is infinite dimensionial.
\end{lemma}

\begin{proof}
  This follows from the  proof of Theorem 1 and a minor modification in the proof of Lemma 7 in \cite{GaHaHe17}. This modification consists of replacing $K\varphi(z)$ by
  $\psi_1(z)+K\psi_2(z)$, and using the boundedness from above of $\psi_2$ to show that the constructed function $u$ belongs to $L^2_{e^{-\psi_1}}(\Omega)$.
\end{proof}

\begin{proof}[Proof of $(c)\Rightarrow(d)$ in Theorem~\ref{T:main}]
  The proof is done by contraposition. 
  Thus, we assume that there exists a bounded function $\psi\in\mathcal{C}^\infty(\p\setminus K)$ such that $i\partial\bar{\partial}\psi\geq \omega$ holds on $\p\setminus K$ for some volume form $\omega$ on $\p$. Note that, if $K$ was empty, then it would follow that there is a bounded, nonconstant subharmonic function on $\mathbb{C}$. Hence, we may assume that $K$ is nonempty. As before, we choose a  $q\in K$ and coordinates such that $q=[0:1]$
  
  Next, we note that, as in Subsection~\ref{SS:CT}, it suffices to show that $A^2_{\phi_k}(D_*)$, for $D_*:=\p\setminus K$, is infinite dimensional for each $k\in\Z$. 
 For that, we use Lemma~\ref{L:GaHaHemodification}. In particular, we set $\wt\psi_2=\psi\circ\varphi_z$. Then, clearly, $\wt\psi_2\in\mathcal{C}^\infty(D_*)$ is a bounded, strictly subharmonic function.
 Moreover, it follows that there is a constant $c_1>0$ such that
 \begin{align*}
   \triangle\wt\psi_2(z)\geq c_1(1+|z|^2)^{-2}\;\;\text{ for }\;\;z\in D_*.
 \end{align*}
 Next, set  $\psi_1:=-\ln(\phi_k)$. 
Then $\psi_1\in\mathcal{C}^\infty(D_*)$ and
 \begin{align}\label{E:Laplacemetric}
   \triangle\psi_1(z)=4(k+2)(1+|z|^2)^{-2}
 \end{align}
 on $D_*$. 
 Hence, there exists a constant $c>0$ such that
 \begin{align*}
   \triangle(\psi_1+c \wt\psi_2)>0\;\;\text{on }D_*.
 \end{align*}
Setting $\psi_2=c\wt\psi_2$,  
 it follows that the hypotheses of Lemma~\ref{L:GaHaHemodification} are satisfied for any open set $U\Subset D_*$. Thus, $A^2_{\phi_k}(D_*)$ is infinite dimensional.
\end{proof}

\subsection{Proof of $\mathbf{(d)}$ implies $\mathbf{(a)}$}\label{SS:(d)=>(a)}
\begin{proof}
This proof is also done by contraposition. That is, we assume that $K\subset\p$ is nonpolar. Since $K$ is nonempty, we may construct $\psi$ in a chart first. As before, we let $q\in K$ and choose coordinates  
such that $q=[0:1]$. We set $K_*=\varphi_z(K)$ and $D_*=\mathbb{C}\setminus K_*$.  
We shall use a function which was initially constructed in the proof of Proposition 5.1 in \cite{GaLeRa21}. 
For that, let $G\subset K_*$ be a nonpolar, compact set and  $\nu=\nu_{G}$ be the equilibrium measure of $G$. The associated potential function
 $p=p_{G}:\mathbb{C}\longrightarrow[-\infty,\infty)$ is defined as
 \begin{align*}
   p(z)=\int_{\mathbb{C}}\ln|z-w|\;d\nu(w)\;\;\text{for }z\in\mathbb{C}.
 \end{align*}
 Since $p$ is harmonic on $G^c$, see for instance \cite[Theorem 3.1.2]{Ra95}, it follows that
 $\triangle e^{-p}=e^{-p}|\nabla p|^2$ on $G^c$.
 Thus, $e^{-p}$ is subharmonic on $G^c$, and strictly subharmonic at all points in $G^c$ at which the gradient of $p$  is nonvanishing.
 It follows from the proof of Proposition~5.1 in \cite{GaLeRa21} that $|\nabla p|$ is strictly positive outside a sufficiently large disc containing $G$. In particular, there exist constants $\tau_1, R>0$ such that 
 $G\subset\mathbb{D}(0,R)$ 
 and
$$|\nabla p(z)|>\frac{\tau_1}{|z|}\;\;\text{for } z\in\mathbb{D}(0,2R)^c.$$ 
Furthermore, since $\nu$ is an equilibrium measure, one obtains for $z\in\mathbb{D}(0,2R)^c$ that
\begin{align*}
  p(z)=\int_{\mathbb{C}}\ln|z-w|\;d\nu(w)\leq\ln(3|z|/2).
\end{align*}
Therefore,
$$\triangle e^{-p(z)}\geq\frac{2}{3}\frac{\tau_1^2}{|z|^3} \;\;\text{for }z\in\mathbb{D}(0,2R)^c.$$
Strict subharmonicity on all of $G^c$ may now be achieved by adding to $e^{-p}$ a particular compactly supported function which is strictly subharmonic on $\mathbb{D}(0,2R)$. For instance, let $\chi\in\mathcal{C}^\infty_0(\mathbb{D}(0,R'))$ for some $R'>2R$ such that $\chi(z)=|z|^2$ on $\mathbb{D}(0,2R)$. Then, for $\epsilon>0$ sufficiently small, there exist $\tau_2,\tau_3>0$ such that the function
$\psi_*(z):=e^{-p(z)}+\epsilon\chi(z)$ satisfies
\begin{align}\label{E:Laplacepsi2}
  \triangle\psi_*(z)\geq 
  \left\{
     \begin{array}{lr}
       \frac{\tau_2}{|z|^3} & \text{for } z\in\mathbb{D}(0,2R)^c\\
       \tau_3 & \text{for } z\in\mathbb{D}(0,2R)
     \end{array}
   \right. .
\end{align}
It follows that $\psi:=\varphi_z^{-1}(\psi_*)$ satisfies $i\partial\bar{\partial}\psi\geq \omega$ on $\p\setminus K$ for some volume form $\omega$ on $\p$.
Moreover, Frostman's theorem implies that $e^{-p}$ is bounded on $G^c$, and hence, $\psi_*$ is bounded on $G^c$. That is, $\psi$ is bounded on $\p\setminus K$, and hence satisfies  $(d)$ in 
Theorem~\ref{T:main}.
\end{proof}


\subsection{A corollary to Lemma~\ref{L:GaHaHemodification} on weighted Bergman spaces}\label{SS:Le}

A variation of Lemma~\ref{L:GaHaHemodification} yields an analogue of Theorem~\ref{T:WC} for the Bergman space $A^2_{e^{-\psi}}(\Omega)$ for $\psi$ subharmonic on $\Omega$, see 
Corollary~\ref{C:LeThesisminusK} below. 

The dimension of $A^2_{e^{-\psi}}(\C)$ for $\psi$ subharmonic on $\C$ is completely described by Borichev, Le, and Youssfi in \cite[Theorem 2.6]{Le21} in terms of the continuous 
part $(\mu_\psi)^c$  of the Riesz measure $\mu_\psi$ of $\psi$:
\begin{itemize}
  \item[$(a)$] $\dim A^2_{e^{-\psi}}(\C)<\infty$ iff $(\mu_\psi)^c(\C)<\infty$,
  \item[$(b)$] if $(\mu_\psi)^c(\C)<\infty$, then  
    \begin{align}\label{E:Le}
      \dim A^2_{e^{-\psi}}(\C)=\left\lceil\frac{(\mu_\psi)^c(\C)}{4\pi}\right\rceil,
    \end{align}  
  where $\lceil x\rceil$ is the largest integer less than $x$ for $x>0$ and $\lceil 0 \rceil=0$.
\end{itemize}  

Recall that $ A^2(\p;\hol(k))$ is isomorphic to $A^2_{\phi_k}(\C)$, and that Lemma~\ref{L:Bergman} states that $\dim A^2(\p, \hol(k))=\max\{0,k+1\}$. Using \cite[Theorem 2.6]{Le21}, we can now 
rediscover Lemma~\ref{L:Bergman} for $k\geq-2$. In particular, for $k\geq -2$, the function $\psi:=-\ln(\phi_k)$ is subharmonic since
$$(\triangle\psi)(z)=4(k+2)(1+|z|^2)^{-2}.$$
It follows that $(\mu_\psi)^c=4(k+2)(1+|z|^2)^{-2}\lambda$, so that $(\mu_\psi)^c(\C)=4\pi(k+2)$. By \eqref{E:Le}, it follows that 
$\dim A^2_{\phi_k}(\C)=\max\{0,k+1\}$.

As in the case of unweighted Bergman spaces, the dimension of a weighted Bergman space for an open set in $\mathbb{C}$ may be determined through the polarity property of its complement as follows.

\begin{cor}\label{C:LeThesisminusK}
  Let $K\subset\mathbb{C}$ be a closed set, $\psi$ a subharmonic function on $\mathbb{C}$. Then the following hold.
  \begin{itemize}
    \item[(1)] If $K$ is nonpolar, then $A^2_{e^{-\psi}}(\C\setminus K)$ is infinite dimensional.
    \item[(2)] If $K$ is polar, then $A^2_{e^{-\psi}}(\C\setminus K)$ is isomorphic to $A^2_{e^{-\psi}}(\C)$.
  \end{itemize}
\end{cor}

\begin{proof}
  To prove (1), we note first that Lemma~\ref{L:GaHaHemodification} still holds if $\psi_1$ is subharmonic, not necessarily smooth, on $\Omega$ as long as the open set $U\Subset\Omega$ may be chosen such 
  that $e^{-\psi_1}$ is 
  integrable on $U$. 
  
  Next, we note that Propositions 2.1 and 2.2 in \cite{Jucha12} imply that $e^{-\psi}$ is integrable except near finitely many points. Thus, we may choose an open set $U$ in $\C\setminus K$ such that 
  $e^{-\psi}$ is integrable on $U$. Now, choose $\psi_2$ equal to $\psi_{*}$ as defined above \eqref{E:Laplacepsi2}. Then $\psi_2$ is a bounded above, strictly subharmonic, smooth function on $\C\setminus K$.
  Thus, with $\psi_1:=\psi$, the claim follows from the above mentioned variation of Lemma~\ref{L:GaHaHemodification}

  For the proof of (2) we first note that $\psi$ is upper semicontinuous, which implies that $\psi$ is bounded from above on any compact subset of $\mathbb{C}$.  Write $\Omega$ for $\C\setminus K$, and let $f\in A^2_{e^{-\psi}}(\Omega)$. It then follows that
  $f\in A^2(\Omega\cap\mathbb{D}(0,R))$ for any $R>0$. We now may proceed as in the paragraph following \eqref{E:metric} and construct a function $F\in\mathcal{O}(\C)$ such that
  $F|_{\Omega}=f$. Since $K$ is of Lebesgue measure $0$, it then follows that $\|F\|_{L^2_{e^{-\psi}}(\C)}=\|f\|_{L^2_{e^{-\psi}}(\Omega)}$. Thus,
  $A^2_{e^{-\psi}}(\C\setminus K)$ is isomorphic to $A^2_{e^{-\psi}}(\C)$.
\end{proof}



\section{Finite dimensional Bergman spaces in $\mathbb P^2$}\label{S:counter}

In this section, we show that the dichotomy displayed by Bergman spaces on $\mathbb P^1$ does not hold in higher dimensional projective spaces. In particular, for every holomorphic line bundle $L$ on $\mathbb P^2$, there exists a domain $D\subset \mathbb P^2$ such that the dimension of the Bergman space $A^2(D;L)$ is finite, but strictly larger than that of the space of global holomorphic sections of $L$. The examples in this section are entirely motivated by Wiegerinck's examples in \cite{Wi84}.

In analogy with $\p$, we use the coordinates $\zt=[\zt_0:\zt_1:\zt_2]$ on $\mathbb P^2$. The open set $U^{(z,w)}=\mathbb P^2\setminus\{\zt:\zt_2=0\}$ is identified with $\C^2$ via
	\bes
		\varphi_{(z,w)}:\zt=[\zt_0:\zt_1:\zt_2]\mapsto (z,w)=
							\left(\frac{\zt_0}{\zt_2},\frac{\zt_1}{\zt_2}\right),
	\ees
and $(z,w)$ are referred to as the local coordinates on the affine chart $U^{(z,w)}$. Any holomorphic line bundle $\Pi:L\rightarrow\mathbb P^2$ is of the form $\hol(k)$, for some $k\in\Z$, where $\hol(k)$ is the line bundle associated to the divisor $k\{\ell\}$ for any fixed hyperplane $\ell\subset\mathbb P^1$. In the local coordinates $(z,w)$, a global section of $\hol(k)$ is given by a triplet, $s=(s_1,s_2,s_3)$, of holomorphic functions on $\C^2$ such that 
	\be\label{E:sectionsP2}
		s_1(z,w)=
		\begin{cases}
			z^{k}s_2(1/z,w/z),& \text{if}\ (z,w)\in\C^2\setminus\{z=0\},\\
						w^{k}s_3(z/w,1/w),& \text{if}\ (z,w)\in\C^2\setminus\{w=0\}.
		\end{cases}
			\ee
Similarly, a Hermitian metric on $\hol(k)$ is given by a triplet, $h=(h_1,h_2,h_3)$, of smooth positive functions on $\C^2$ that satisfy compatibility conditions analogous to \eqref{E:sectionsP2}. In view of Lemma~\ref{L:choice}, we fix the following smooth volume form on $\mathbb P^2$, and Hermitian metric on $\hol(k)$, $k\in\Z$, respectively, in the local coordinates $(z,w)$ on $U^{(z,w)}$:
\beas
	\om_{\text{FS}}(z,w)&=&(1+|z|^2+|w|^2)^{-3}\,dz\wedge d\zbar\wedge dw\wedge d\wbar\\
h_1(z,w)&=&(1+|z|^2+|w|^2)^{-k}.
\eeas
Finally, recall the following description of the space of global sections of $\hol(k)$, see \cite[p. 81]{Varolin11}.

\begin{lemma}\label{L:Bergman2} Let $k\in\Z$. In the local coordinates $(z,w)$ on the affine chart $U^{(z,w)}$, 
	\bes
		A^2(\mathbb P^2;\hol(k))=\Gamma(\mathbb P^2;\hol(k))=\begin{cases}
				\text{span}\{z^pw^q:(p,q)\in\N^2,\ p+q\leq k\},& \text{if}\ k\geq 0,\\
				\{f\equiv 0\},& \text{if}\ k< 0.	
				\end{cases}
	\ees
In particular, $\dim A^2(\mathbb P^2;\hol(k))=\dim\Gamma(\mathbb P^2;\hol(k))=\max\left\{0,\frac{(k+1)(k+2)}{2}\right\}$. 
\end{lemma}

\begin{theorem}\label{T:WiegP2} Given $k\in\Z$, there exists a domain $\Om_{k}\subset\mathbb P^2$ such that 
	\bes
		\dim A^2(\mathbb P^2;\hol(k))<\dim A^2(\Om_{k};\hol(k))<\infty.
	\ees
In particular, $\dim A^2(\Om_{k};\hol(k))=1$ for $k<-2$, and 	
$\dim A^2(\Om_{k};\hol(k))=(k+3)(k+4)/2$ for $k\geq -2$.
\end{theorem}
\begin{proof} It suffices to produce a domain in the chart $U^{(z,w)}$, i.e., a domain $\Om_{k}\subset\C^2$ such that $\dim A^2_{\phi_k}(\C^2)<\dim A^2_{\phi_k}(\Om_{k})<\infty$, where
	\bes
		\phi_k(z,w)=(1+|z|^2+|w|^2)^{k+3},\qquad (z,w)\in\C^2.
	\ees
 The desired domain will be of the form $B\cup X_\ell\cup Y\cup Z_m$, for appropriately chosen $\ell,m\in\N$, where 
	\beas
		B&=&\left\{(z,w)\in\C^2:\max\{|z|,|w|\}<2\right\},\\
		X_\ell&=&\left\{(z,w)\in\C^2:|z|>\sqrt{2},\ |w|<1/|z|^\ell\right\},\quad \ell\in\N,\\
		Y&=&\left\{(z,w)\in\C^2:|w|>\sqrt{2},\ |z|<1/|w|\right\},\\
		Z_m&=&\left\{(z,w)\in\C^2:|z|^2+|w|^2>2,\ \big||z|-|w|\big|<\frac{1}{(|z|+|w|)^m}\right\},\quad m\in\N\setminus\{0,1\}.
	\eeas
In the rest of the proof, $A^2_{k,\text{mon}}(\Om)$ denotes the set of monomials in $A^2_{\phi_k}(\Om)$, for any $\Om\subset\C^2$. We claim that 
	\bea
		A^2_{k,\text{mon}}(B)&=&\{z^pw^q:(p,q)\in\N^2\},\label{E:B}\\
		A^2_{k,\text{mon}}(X_\ell)&=&\{z^pw^q:(p,q)\in\N^2,\ p-\ell q\leq \ell+k +1\}, \label{E:X}\\
		A^2_{k,\text{mon}}(Y)&=&\{z^pw^q:(p,q)\in\N^2,\ q-p\leq k+2\}, \label{E:Y}\\
		A^2_{k,\text{mon}}(Z_m)&=&\{z^pw^q:(p,q)\in\N^2,\ 2(p+q)\leq m+2k+2\}. \label{E:Z}
	\eea
Assuming \eqref{E:B}-\eqref{E:Z} for the moment, set
	\be\label{E:dom}
		\Om_{k}=
		\begin{cases}
			B\cup X_1\cup Y\cup Z_2,\ & \text{if}\ k\geq -2,\\
			B\cup X_{1-2(k+2)}\cup Y\cup Z_{-2(2k+3)},\ & \text{if}\ k<-2.
		\end{cases}
	\ee
Since $\Om_k$ is Reinhardt for any $k\in\Z$, $A^2_{k,\text{mon}}(\Om_k)$ is a basis for $A^2_{\phi_k}(\Om_k)$. Thus, $\dim A^2_{\phi_k}(\Om_k)$ is the cardinality of  $A^2_{k,\text{mon}}(\Om_k)$. From \eqref{E:dom}, if $k\geq -2$, 
	\bes
		A^2_{k,\text{mon}}(\Om_k)=\{z^pw^q:(p,q)\in\N^2,\ |p-q|\leq k+2,\ p+q\leq k+2\},
	\ees 
which is of cardinality  $\frac{(k+3)(k+4)}{2}$. On the other hand, if $k<-2$, i.e., $k+2=-t$ for some $t\geq 1$, then 
\bes
A^2_{k,\text{mon}}(\Om_k)=\left\{z^pw^q:(p,q)\in\N^2,\ \frac{p-t}{1+2t}\leq q\leq p-t,\ p+q\leq t\right\}=\{z^t\}.
\ees
Thus, recalling Lemma~\ref{L:Bergman2}, $\dim A^2_{\phi_k}(\C^2)<\dim A^2_{\phi_k}(\Om_{k})<\infty$, in either case. 

It now remains to prove \eqref{E:B}-\eqref{E:Z}. We use the notation $L\approx M$ for $L,M\in\mathbb{R}$ to mean that there exist constants $c,d>0$ such that
$cM\leq L\leq dM$.
 Since \eqref{E:B} is clear, and \eqref{E:Y} follows from \eqref{E:X}, we only need to prove \eqref{E:X} and \eqref{E:Z}.
To find the monomials which are contained in $A^2_{\phi_k}(X_\ell)$, consider
	\beas
		\wt X_\ell&=&\left\{(r,s): r>\sqrt{2}, 0< s <1/r^\ell\right\}.
\eeas
For $p,q,\ell\in\N$, we get
$$||z^pw^q||^2_{A^2_{\phi_k}(X_\ell)}=(2\pi)^2\int_{\wt X_\ell}
		\frac{r^{2p+1}s^{2q+1}}{(1+r^2+s^2)^{3+k}}\;ds\;dr=:(2\pi)^2\cdot \mathcal{J}.$$
To determine the range of $p$ and $q$ for which $\mathcal{J}$ is finite for a given $\ell$, we consider $\wt X_\ell$ in polar coordinates, i.e.,
$$\wt X_\ell=\{(R\cos\theta, R\sin\theta): R\cos\theta>\sqrt{2},\; R^{\ell+1}\sin\theta(\cos\theta)^\ell<1, \theta\in(0,\pi/2)\},$$
and introduce
$$\mathcal{J}(a,g):=\int_{a}^{\infty}\int_{0}^{g(R)}\frac{R^{2p+2q+2}}{(1+R^2)^{3+k}}(\cos\theta)^{2p+1}(\sin\theta)^{2q+1} R\;d\theta\;dR,$$		
where $a>0$ is a constant and $g$ is a positive, continuous function. We notice that $\theta\in(0,\pi/4)$ on $\wt X_\ell$. It then follows that
$$ \mathcal{J}(2,\arcsin(R^{-\ell-1}))\leq \mathcal{J}\leq\mathcal{J}(\sqrt{2}, \arcsin(\sqrt{2}^\ell R^{-\ell-1})),$$
 because, firstly, any point in $\wt X_\ell$ satisfies $R>\sqrt{2}$ and $\theta\in(0,\arcsin(\sqrt{2}^\ell R^{-\ell-1}))$, and, secondly,
 any point with $R>2$ and $\theta\in(0,\arcsin(R^{-\ell-1}))$ is a point in $\wt X_\ell$. Hence, it remains to estimate $\mathcal{J}(a,\arcsin(cR^{-\ell-1})$ for $c\in\{1,\sqrt{2}^\ell\}$. Again,
 since $\theta\in(0,\pi/4)$ on $\wt X_\ell$, it follows that
 \begin{align*}
   \mathcal{J}(a,\arcsin(c R^{-\ell-1}))
   &=\int_{a}^\infty\frac{R^{2p+2q+2}}{(1+R^2)^{3+k}}\left(\int_{0}^{\arcsin (cR^{-\ell-1})} (\cos\theta)^{2p+1}(\sin\theta)^{2q+1}\;d\theta\right)R\;dR\\
&\approx \int_{a}^\infty\frac{R^{2p+2q+2}}{(1+R^2)^{3+k}}\left(\int_{0}^{\arcsin (cR^{-\ell-1})} \cos\theta (\sin\theta)^{2q+1} \;d\theta\right)R\;dR\\
&\approx \int_{a}^\infty R^{2p+2q+3-(\ell+1)(2q+2)-2(3+k)}\;dR,
 \end{align*}
which converges if and only if $2p+2q+3-(\ell+1)(2q+2)-2(3+k)<-1$. It follows that $z^pw^q\in A^2_{\phi_k}(X_\ell)$ if and only if
	\be\label{E:cond1}
		p-\ell q\leq k+\ell+1.
	\ee	
In a similar fashion, we  find the monomials in $A^2_{\phi_k}(Z_m)$. For that, we write
	\beas
	\wt Z_m=\left\{(r,s):r^2+s^2>2,\ |r-s|(r+s)^m<1\right\}
	\eeas
and note that 
for $p,q,m\in\N$ 
	\begin{align*}
		||z^pw^q||^2_{A^2_{\phi_k}(Z_m)}
	=(2\pi)^2\int_{\wt Z_m}\frac{r^{2p+1}s^{2q+1}}{(1+r^2+s^2)^{3+k}}\;dr\;ds:=(2\pi)^2\cdot\mathcal{I}.	
	\end{align*}
As before, we introduce polar coordinates on $\wt Z_m$. It is straightforward to check that then
\begin{align*}
  \wt Z_m
  &= \left\{(R\cos\theta,R\sin\theta):R>\sqrt{2},\, \left|\sin\left(\textstyle\frac{\pi}{4}-\theta\right)\right|
  \cos^m\left(\textstyle\frac{\pi}{4}-\theta\right)<(\sqrt{2}R)^{-m-1},\theta\in(0,\textstyle\frac{\pi}{2})\right\}\\
  &=\left\{\left(R\cos\left(\textstyle\frac{\pi}{4}-\psi\right), R\sin\left(\textstyle\frac{\pi}{4}-\psi\right)\right):
  R>\sqrt{2},\,\left|\sin\psi\right|\cos^m\psi<(\sqrt{2} R)^{-m-1}, \psi\in(-\textstyle\frac{\pi}{4},\textstyle\frac{\pi}{4})
  \right\}.
\end{align*}
We work in the coordinates $(R,\psi)$ to estimate $\mathcal{I}$, utilizing integrals of the form
\begin{align*}
  \mathcal{I}(g):=\int_{\sqrt{2}}^\infty\int_{-g(R)}^{g(R)}
  \frac{R^{2p+2q+3}}{(1+R^2)^{3+k}}
	 \cos^{2p+1}\left(\textstyle\frac{\pi}{4}-\psi\right) \sin^{2q+1}\left(\textstyle\frac{\pi}{4}-\psi\right)  \;d\psi \;dR
\end{align*} 
for some positive, continuous function $g$. Since $\psi\in[-\pi/4,\pi/4)$, it follows that
\begin{align*}
  \mathcal{I}\Bigl(\arcsin\bigl((\sqrt{2}R)^{-m-1}\bigr)\Bigr)\leq \mathcal{I}\leq\mathcal{I}\left(\arcsin\left(\sqrt{2}R^{-m-1}\right)\right).
\end{align*}
Since $m\geq 2$, both trigonometric functions in the definition of $\mathcal{I}(g)$ are strictly positive for $|\psi|\leq|\arcsin(\sqrt{2}R^{-m-1})|$. 
Therefore,
\begin{align*}
  \mathcal{I}\left(\arcsin(cR^{-m-1})\right)\approx
  \int_{\sqrt{2}}^\infty\int_{-\arcsin(cR^{-m-1})}^{\arcsin(cR^{-m-1})}
  \frac{R^{2p+2q+3}}{(1+R^2)^{3+k}}
	  \;d\psi \;dR
\end{align*}
for $c\in\{\sqrt{2},\sqrt{2}^{-m-1}\}$.
Moreover, $|t/2|\leq|\sin t|\leq |t|$ holds for $|t|\leq\pi/4$. Thus
\begin{align*}
  \mathcal{I}\left(\arcsin(cR^{-m-1}\right)&\approx\int_{\sqrt{2}}^\infty R^{2p+2q+3-(m+1)-2(3+k)}dR,
\end{align*}
which converges if and only if 
	\be\label{E:cond2}
		2(p+q)\leq m+2k+2.
	\ee
This completes the proof of \eqref{E:Z}, and thus, of Theorem~\ref{T:WiegP2}.
\end{proof}

\bibliography{referencesBergmanRS}{}
\bibliographystyle{plain}
\end{document}